\documentclass[12pt, a4paper]{article}

\usepackage{amsmath, amsthm, amssymb, url, enumerate}
\usepackage[pdftex]{graphicx}
\usepackage{cases}
\usepackage{epic, eepic, ecltree}
\usepackage{stmaryrd}
\usepackage[T1]{fontenc}
\usepackage{textcomp}
\usepackage{color}
\usepackage[all]{xy}
\usepackage{cases}
\usepackage{stmaryrd}

\usepackage[svgnames]{xcolor}%
\usepackage{tikz}
\usetikzlibrary{intersections, calc}
\usetikzlibrary {arrows.meta}

\newtheorem{definition}{Definition}[section] 
\newtheorem{theorem}[definition]{Theorem}
\newtheorem{proposition}[definition]{Proposition} 
\newtheorem{lemma}[definition]{Lemma}
\newtheorem{remark}[definition]{Remark}
\newtheorem{corollary}[definition]{Corollary}

\newtheorem{example}[definition]{Example}
\renewcommand{\proofname}{Proof.}

\makeatother

\def\CC{\mathbb{C}}

\def\PP{\mathbb{P}}

\newcommand{\wt}{{\rm wt}}

\newcommand{\compl}[1]{\mathbb{C}^{#1}}

\makeatletter
\def\mapstofill@{%
   \arrowfill@{\mapstochar\relbar}\relbar\rightarrow}
\newcommand*\xmapsto[2][]{%
   \ext@arrow 0395\mapstofill@{#1}{#2}}
\makeatother

\begin{document}
\title{\vspace{-3cm}$k$-Wahl chains and cyclic quotient singularities}
\author{Yusuke\ Sato}
\date{}
\maketitle
\thispagestyle{empty}

\noindent

\begin{abstract}
We study two-dimensional cyclic quotient singularities defined by
$k$-Wahl chains, a class of Hirzebruch--Jung continued fractions
obtained inductively starting from $[k+2]$.
This class includes the classical Wahl singularities in the case $k=2$ and
also contains cyclic quotient singularities arising from $k$-generalized
Markov triples.
For singularities defined by $k$-Wahl chains, we prove that the combinatorics
of the continued fraction is encoded in the special representations of the
associated cyclic group.
We also study zero continued fractions on the dual side and obtain
consequences for the deformation theory of these singularities, including
the existence of extremal P-resolutions in the case of $1$-Wahl chains.
\end{abstract}

\noindent

\section{Introduction}

Two-dimensional cyclic quotient singularities form a fundamental class of surface singularities that can be described explicitly by Hirzebruch--Jung continued fractions. 
For a cyclic quotient singularity of type $\frac{1}{r}(1,a)$, where $r$ and $a$ are coprime,
the continued fraction expansion
\[
\frac{r}{a}=[b_{1},\dots,b_{\ell}]
= b_{1}-\cfrac{1}{b_{2}-\cfrac{1}{\ddots-\cfrac{1}{b_{\ell}}}},
\qquad b_{i}\in \mathbb{Z}_{\ge 2}.
\]
gives the self-intersection numbers of the exceptional curves in the minimal resolution. 
Because of this, combinatorial properties of Hirzebruch--Jung continued fractions often reflect geometric properties of the corresponding singularities.

In this paper, we study a distinguished class of cyclic quotient singularities arising from what we call \emph{$k$-Wahl chains}. 
A $k$-Wahl chain is defined by a Hirzebruch--Jung continued fraction obtained inductively from the initial chain $[k+2]$ by repeatedly attaching a $2$ either to the left or to the right, while increasing the opposite end by one.  
When $k=2$, this recovers the classical Wahl chains, and the associated cyclic quotient singularities are the Wahl singularities, which form a distinguished subclass of $T$-singularities and appear naturally in the study of $\mathbb Q$-Gorenstein deformations \cite{KSB, Wahl}.

Another motivation comes from number theory. 
The notion of $k$-Wahl chains is closely related to the theory of a $k$-generalized Markov equation.  
The \emph{$k$-generalized Markov equation} is introduced by Gyoda and Matsushita \cite{GM} as following:
\[
x^{2}+y^{2}+z^{2}+k(xy+yz+xz)=(3+3k)xyz.
\]
Its positive integer solutions are called \emph{$k$-generalized Markov triples}, and the integers appearing in such triples are called \emph{$k$-generalized Markov numbers}. 
The $k$-generalized Markov triple determines a Hirzebruch--Jung continued fraction, and this continued fraction is a $k$-Wahl chain \cite{GMS}. 
Thus, $k$-Wahl chains provide a natural geometric framework containing cyclic quotient singularities arising from $k$-generalized Markov numbers.

Our first main result establishes a representation-theoretic property of cyclic quotient singularities defined by $k$-Wahl chains. 
For a cyclic group $G=\frac{1}{r}(1,a)\subset GL(2,\mathbb C),$
the special McKay correspondence gives a bijection between the exceptional curves in the minimal resolution of $\mathbb C^2/G$ and the nontrivial special representations of $G$ \cite{Rie,Wun1}. 
To a $k$-Wahl chain, we associate its \emph{length encoding}, which records the pattern of successive attachments of $2$'s in the inductive construction. 
On the other hand, from the sequence of special representations of $G$, we define the \emph{length encoding of special representations} (LESR) by taking the run-length encoding of the corresponding differences. 
Our first main theorem shows that these two encodings coincide.

\begin{theorem}
Let $G=\frac{1}{r}(1,a)$ be a cyclic quotient singularity with $\gcd(r,a)=1$, and suppose that the Hirzebruch--Jung continued fraction  $\frac{r}{a}=[b_1,\dots,b_s]$ is a $k$-Wahl chain. 
Then the length encoding of $[b_1,\dots,b_s]$ coincides, up to reversal, with the length encoding of special representations of $G$.
\end{theorem}

This theorem shows that the inductive combinatorics of $k$-Wahl chains is reflected precisely in the arrangement of special representations of the corresponding cyclic group. 
The second theme of this paper is the connection between $k$-Wahl chains and deformation theory via \emph{zero continued fractions}. 
A zero continued fraction is a negative continued fraction
\[
[a_1,\dots,a_s], \qquad a_i\geq 1,
\]
whose value is equal to zero. For a cyclic quotient singularity $X=\mathbb{C}^2/G$ with $G=\frac{1}{r}(1,a)$ and its Hirzebruch--Jung continued fraction $\frac{r}{a}=[b_1,\dots,b_s]$, we set a dual Hirzebruch--Jung continued fraction
\[
\frac{r}{r-a}=[e_1,\dots,e_\ell].
\]
A fundamental result of Christophersen \cite{Christophersen}, revisited from the toric viewpoint by Altmann \cite{Altmann}, states that the P-resolutions of $X$ are in one-to-one correspondence with the admissible zero continued fractions on the dual side. 
Thus, the existence of a P-resolution can be read combinatorially from the existence of a zero continued fraction supported by the dual Hirzebruch--Jung continued fraction. 
Since P-resolutions are in bijection with the irreducible components of the versal deformation space, zero continued fractions provide concrete information on the deformation theory of cyclic quotient singularities \cite{Altmann, Christophersen, KSB}.

In this direction, we study the minimal weight of zero continued fractions for dual $0$- and $1$-Wahl chains.

\begin{theorem}
Let $[b_1,\dots,b_s]$ be a Hirzebruch--Jung continued fraction with $s\geq 2$. Then the following hold:
\begin{itemize}
  \item[{\rm (i)}] If it is a dual $1$-Wahl chain, then $\alpha([b_1,\dots,b_s])=1$.
  \item[{\rm(ii)}] If it is a dual $0$-Wahl chain, then $\alpha([b_1,\dots,b_s])=2$,
\end{itemize}
where $\alpha([b_1,\dots,b_s])$ denotes the minimal weight of an admissible zero continued fraction.
\end{theorem}

In particular, dual \(1\)-Wahl chains give rise to extremal P-resolutions for cyclic quotient singularities corresponding to $1$-Wahl chains.
In this way, cyclic quotient singularities associated with $k$-Wahl chains may be viewed as lying at the intersection of number theory, representation theory, and deformation theory.

The organization of this paper is as follows. 
In Section~2, we review the definition of $k$-Wahl chains and their duals.
In Section~3, we introduce the length encoding of $k$-Wahl chains and compare it with the special representations of the corresponding cyclic group, proving the representation-theoretic characterization above. 
In Section~4, we study zero continued fractions and their weights, and we investigate how dual $k$-Wahl chains give rise to $P$-resolutions. 
This provides a deformation-theoretic interpretation of the class of cyclic quotient singularities defined by $k$-Wahl chains.
\section{$k$-Wahl chains and cyclic quotient singularities}\label{sec:GMsing}

In this section, we explain cyclic quotient singularities associated with $k$-Wahl chains.
Let $a,b,r$ be positive integers and let $\varepsilon$ be a primitive $r$-th root of unity.
Let $G$ be the cyclic group generated by the diagonal matrix$
\begin{pmatrix}
\varepsilon^{a} & 0\\
0 & \varepsilon^{b}
\end{pmatrix}.$
We denote this generator (and also the group) by $\frac{1}{r}(a,b)$.
Then $G$ acts on $\mathbb{C}^{2}$ by $(x,y)\mapsto (\varepsilon^{a}x,\varepsilon^{b}y)$, and the quotient
$\mathbb{C}^{2}/G$ is called a \emph{cyclic quotient singularity of type}
$\frac{1}{r}(a,b)$.

Before stating properties of $k$-Wahl chains, we recall the relation between
two-dimensional cyclic quotients and continued fractions.
Let $G_{r,a}=\frac{1}{r}(1,a)$ with $\gcd(r,a)=1$.
The \emph{Hirzebruch--Jung} continued fraction is
\[
\frac{r}{a}=[b_{1},\dots,b_{\ell}]
= b_{1}-\cfrac{1}{b_{2}-\cfrac{1}{\ddots-\cfrac{1}{b_{\ell}}}},
\qquad b_{i}\in \mathbb{Z}_{\ge 2}.
\]
Then $-b_{1},\dots,-b_{\ell}$ are the self-intersection numbers of the exceptional curves
$E_{1},\dots,E_{\ell}$ in the minimal resolution of $\mathbb{C}^{2}/G_{r,a}$.

Motivated by this, we introduce the following special class of Hirzebruch--Jung fractions.

\begin{definition}\upshape \label{def:kWahl}
Let $k$ be a non-negative integer. A \emph{$k$-Wahl chain} is defined inductively by:
\begin{enumerate}
\item $[k+2]$ is a $k$-Wahl chain.
\item If $[b_{1},\dots,b_{\ell}]$ is a $k$-Wahl chain, then so are
\[
[b_{1}+1,b_{2},\dots,b_{\ell},2]
\qquad\text{and}\qquad
[2,b_{1},\dots,b_{\ell-1},b_{\ell}+1].
\]
\end{enumerate}
\end{definition}

\begin{example}\upshape \label{ex:kWahl}
When $k=0$, the fractions $[2]$, $[3,2]$, $[4,2,2]$, and $[2,3,3]$ are $0$-Wahl chains.
Moreover, if $[b_{1},\dots,b_{\ell}]$ is a $k$-Wahl chain, then $[b_{\ell},\dots,b_{1}]$ is also a $k$-Wahl chain.
\end{example}

\begin{remark}\upshape \label{rem:Tsing}
The notion of $k$-Wahl chains originally appeared (for $k=2$) in the work of Wahl \cite{Wahl},
Koll\'ar, and Shepherd-Barron \cite{KSB}. In particular, the cyclic quotient singularities determined by
$2$-Wahl chains are Wahl singularities, which belong to the class of $T$-singularities.
A two-dimensional quotient singularity is called a \emph{$T$-singularity} if it admits a
$\mathbb{Q}$-Gorenstein smoothing.
Moreover, if a quotient singularity is a $T$-singularity, then it is either a rational double point
or a cyclic quotient singularity of type $\frac{1}{dm^{2}}(1,adm-1)$ with $\gcd(a,m)=1$ and $m>1$.
The case $d=1$ (and not a rational double point) is called a \emph{Wahl singularity}.
\end{remark}


\begin{theorem}[\cite{GMS}]\label{thm:GMimpliesWahl}
Let $t\in(0,1)$ be a reduced fraction, and let $m_{k,t}$ be the associated $k$-GM number.
Let $u^{+}_{k,t}$ be the corresponding characteristic number.
Then the Hirzebruch--Jung continued fraction $\frac{m_{k,t}}{u^{+}_{k,t}}$
is a $k$-Wahl chain.
\end{theorem}

 The relationship between Wahl singularities and the Markov equation, namely the $0$-GM equation, has been studied by Hacking and Prokhorov~\cite{HP}, Perling~\cite{Per}, and Urz\'ua and Z\'u\~niga~\cite{UZ}. These previous works consider the weighted projective surface $\PP(a^2,b^2,c^2)$ associated with an ordinary Markov triple, that is, a $0$-GM triple $(a,b,c)$, and it is known that this weighted projective surface has only Wahl singularities.  On the other hand, Theorems~\ref{thm:GMimpliesWahl} and \ref{thm:square} show that the cyclic quotient singularities associated with the squares of $0$-GM triples are Wahl singularities.

\begin{theorem}[\cite{GM}]\label{thm:square}
If $(a,b,c)$ is a $0$-GM triple, then $(a^{2},b^{2},c^{2})$ is a $2$-GM triple.
Conversely, if $(A,B,C)$ is a $2$-GM triple, then $(\sqrt{A},\sqrt{B},\sqrt{C})$ is a $0$-GM triple.
\end{theorem}

\begin{remark}\upshape
The converse of Theorem \ref{thm:GMimpliesWahl} does not hold in general: for example,
$\frac{10}{3}=[4,2,2]$ is a $0$-Wahl chain, but $10$ is not a $0$-GM number.
\end{remark}

\begin{proposition}\label{prop:kwahldual}
Let 
$\frac{r}{a} = [b_1, \dots, b_s]$ be a $k$-Wahl chain for some non-negative integer $k$. Then we have
\[
\frac{r}{\,r-a-k\,} = [b_s, \dots, b_1].
\]
In particular, it follows that $
a(r-a-k) \equiv 1 \pmod{r}$.
\end{proposition}
\noindent

\begin{proof}
We argue by induction on the length of the Hirzebruch--Jung continued fraction.  
First, note that $\frac{k+2}{1} = [k+2]$
satisfies the claim, since $k+2-1-k = 1$.

Now suppose that $\frac{r}{a} = [b_1, \dots, b_s]$
is a $k$-Wahl chain and that the statement holds for this continued fraction. 
We consider the extension $[b_1+1, b_2, \dots, b_s, 2]$.
By the induction hypothesis we have
\[
[2, b_s, \dots, b_1] = 2-\frac{r-a-k}{r}
= \frac{r+a+k}{r}.
\]

From this, we deduce
\[
[b_1, \dots, b_s, 2] = \frac{r+a+k}{r'},
\]
where $r'$ denotes the integer such that $r'\cdot r \equiv 1 \pmod{r+a+k}$.
It leads to
\[
[b_1+1, b_2, \dots, b_s, 2] 
= 1 + [b_1, \dots, b_s, 2]
= \frac{r' + r + a + k}{r'}.
\]

On the other hand, since
\[
[1+b_1, b_2, \dots, b_s] 
= 1 + \frac{r}{a} 
= \frac{r+a}{a},
\]
we have
\[
[2, b_s, \dots, b_2, b_1+1] = 2-\frac{a'}{r+a}=\frac{2r+2a-a'}{r+a},
\]
where $a'$ denotes the integer satisfying $a' \cdot a \equiv 1 \pmod{r+a}$.
By the basic property of Hirzebruch--Jung continued fractions, we have $2r+2a-a' = r'+r+a+k.$
Hence
\[
[2, b_s, \dots, b_2, b_1+1]
=\frac{r'+r+a+k}{r+a}.
\]
It remains to check that the denominator is
\[
(r'+r+a+k)-r'-k = r+a,
\]
which is immediate. Therefore the claim follows.\\
A similar argument applies to $[2,b_1, \dots, b_{s-1}, b_s+1]$.
This completes the induction and proves the proposition.

\end{proof}

\begin{definition}\label{def:dualkwahlchainstwo}\upshape
Let $k$ be a non-negative integer. We define a \emph{dual $k$-Wahl chain} as follows.
\begin{itemize}\setlength{\leftskip}{-15pt}
\item [$(i)$] The Hirzebruch-Jung continued fraction with $k+1$ consecutive $2$s $[2,\dots,2]$ is a dual $k$-Wahl chain.
\item[$(ii)$] If $[b_1,\dots,b_l]$ is a dual $k$-Wahl chain, then $ [b_1+1,b_2,\dots,b_\ell,2]$ and $[2,b_1,\dots, b_{\ell-1},b_{\ell}+1]$ are also dual $k$-Wahl chains.
\end{itemize}
\end{definition}

\begin{remark}\upshape \label{rem:dualkwahl}
Any dual $k$-Wahl chain arises as the Hirzebruch--Jung dual of some $k$-Wahl chain. 
In other words, if $[b_1,\dots,b_{\ell}] = \frac{r}{a}$ is a $k$-Wahl chain, then $\frac{r}{r-a} = [a_1,\dots,a_s]$
is a dual $k$-Wahl chain.

In general, the Hirzebruch--Jung dual of a continued fraction does not coincide with its reversal. 
However, for a $0$-Wahl chain, they do coincide by Proposition~\ref{prop:kwahldual}: namely,
$
\frac{r}{r-a} = [b_{\ell},\dots,b_1].
$

\end{remark}

\section{Length encoding}
 In this section, we study the special representations associated with the cyclic groups determined by $k$-Wahl chains.
By the special McKay correspondence, these special representations are in one-to-one correspondence with the exceptional curves appearing in the minimal resolution of the corresponding cyclic quotient singularities.
Using this geometric correspondence, we introduce the relationship between the length encoding of a $k$-Wahl chain and that of the special representations.
\subsection{Special representations}

Let $G_{r,a} = \tfrac{1}{r}(1,a)$ be a finite cyclic group, and let 
$g = \tfrac{1}{r}(1,a)$ denote its generator. 
For each integer $i$, we define the character $\chi_i : g \mapsto \varepsilon^i$, 
where $\varepsilon$ is a primitive $r$-th root of unity, and let $\rho_i$ 
be the one-dimensional representation associated to $\chi_i$. 
Since $G_{r,a}$ is abelian, all its irreducible representations are one-dimensional. 
Among these irreducible representations, we now introduce the notion of 
\emph{special representations}.
Originally, special representations were defined for small finite subgroups of 
${\rm GL}(2,\CC)$ (see \cite{Rie}, \cite{Wun1}). For the purposes of this work, we shall restrict the definition to finite cyclic groups.

\begin{definition}\upshape
Let $G$ be a finite cyclic group and $\rho$ a one-dimensional representation of $G$ with representation space $V_\rho$.
We say that $\rho$ is a nontrivial \emph{special representation} of $G$ if 
$(\CC[x,y]\otimes V_\rho)^G$ is generated by two elements as a module over $\CC[x,y]^G$.
\end{definition}

There is a one-to-one correspondence between nontrivial special representations 
and the exceptional curves appearing in the minimal resolution of the 
cyclic quotient singularity $\CC^2/G$. 
This correspondence is known as the \emph{special McKay correspondence} 
(\cite{Wun2}).

A monomial \(x^m y^n\) in $\CC[x,y]$ transforms by the character \(\chi_{m+an}\).
Equivalently, the weight of \(x^m y^n\) is
\[
\wt(x^m y^n)\equiv m+an \pmod r.
\]
In particular, the monomials \(x^i\) and \(y^j\) define the same character if and only if $i\equiv aj \pmod r$.
Therefore, if an exceptional curve is described by a monomial ratio $x^i:y^j$,
then the corresponding special representation is $\rho_i$.

From the toric point of view, the cyclic quotient singularity
\(\mathbb{C}^2/G_{r,a}\) is the affine toric surface associated with the cone
\(\sigma\subset N_\mathbb{R}\), where
\[
N=\mathbb{Z}^2+\frac{1}{r}(1,a)\mathbb{Z} \ , \textrm{and} \  \sigma=\mathbb{R}_{\geq 0}(1,0)+\mathbb{R}_{\geq 0}(0,1).
\]
The minimal resolution of $\compl{2}/G_{r,a}$ is obtained by a simplicial subdivision whose rays pass through the lattice points on the Newton boundary, and each lattice point on the Newton boundary
corresponds to an exceptional divisor.
If \(i\equiv aj \pmod r\), then we have $\frac{1}{r}(j,i)\in N $.
Moreover, this lattice point is precisely the toric point corresponding to the
exceptional divisor associated with the special representation \(\rho_i\).

\begin{remark}\upshape \label{rem:SR-remainders}
Let $\frac{r}{a}=[b_1,\dots,b_s]$
be the Hirzebruch--Jung continued fraction of the cyclic quotient singularity of type
\(\frac{1}{r}(1,a)\).
The integers \(a_0,\dots,a_{s+1}\) are defined by
\[
a_0:=r,\qquad a_1:=a,\qquad a_{i+1}:=b_i a_i-a_{i-1}\ \ (i=1,\dots,s).
\]
Then
\[
r=a_0>a_1>\cdots>a_s=1>a_{s+1}=0,
\]
and the indices of the special representations are precisely these intermediate
remainders:
\[
\rho_{a_1},\rho_{a_2},\dots,\rho_{a_s}.
\]
In particular, the subscripts of the special representations are not the coefficients
\(b_i\) of the Hirzebruch--Jung continued fraction, but the remainder sequence
associated with it.
\end{remark}

\begin{definition}\upshape
Let $G_{r,a} = \tfrac{1}{r}(1,a)$ be a cyclic group. Denote by 
$\rho_{a_1}, \dots, \rho_{a_s}$ the special representations of $G_{r,a}$, 
where $r > a_1 > a_2 > \cdots > a_s > 0$. 
For each $i = 1, \dots, s-1$, set $d_i := a_i - a_{i+1}$.  

The \emph{special representation length encoding} (LESR) of $G_{r,a}$ is defined 
as the sequence obtained from $(d_1, \dots, d_{s-1})$ by replacing each 
maximal block of consecutive equal integers with its length.  
In particular, when $G_{r,a}$ admits only one special representation, that is $G_{r,1}$, we set 
$\mathrm{LESR}(G_{r,a}) = \varnothing$.
\end{definition}

\begin{example}\label{ex29}
Let $G=G_{29,12}$. Then the special representations of $G_{29,12}$ are 
$\rho_{12}, \rho_{7}, \rho_{2}, \rho_{1}$, corresponding to the monomial ratios 
$x^{12}:y$, $x^{7}:y^3$, $x^{2}:y^5$, and $x:y^{17}$, respectively. 
These special representations correspond bijectively to the exceptional curves 
$\mathbb{P}^1$ in the minimal resolution of $\mathbb{C}^2/G_{29,12}$.

From this we obtain the sequence 
\[
d_1 = 12-7 = 5, \quad d_2 = 7-2 = 5, \quad d_3 = 2-1 = 1.
\]
Hence the sequence $(d_1,d_2,d_3) = (5,5,1)$.  
Applying length encoding, we compress this into 
\[
\mathrm{LESR}(G_{29,12}) = (2,1).
\]

\end{example}

Note that, by definition of LESR, reversing the order of the continued fraction 
also reverses its length encoding. Namely, if the LESR of 
$[b_1, \dots, b_s]$ is $(c_1, \dots, c_m)$, then the LESR of 
$[b_s, \dots, b_1]$ is $(c_m, \dots, c_1)$.

\begin{lemma}\label{lem:LErule}
Let $\frac{r}{a}=[b_1,\dots,b_s]$, and write the special representations of $G_{r,a}$ as 
 $\rho_{a_1},\dots,\rho_{a_s}$, where $a_1=a$ and  $a_s=1$. For each $i=0,\dots,s$, set $d_i=a_{i}-a_{i+1}$, where $a_0=r, a_{s+1}=0$. Then $b_i=2$ if and only if $d_i-d_{i-1}=0$ for $i=1,\dots, s$.
\end{lemma}

\noindent
\proofname 
By Remark \ref{rem:SR-remainders}, we have $b_i\cdot a_i=a_{i+1}+a_{i-1}$ for $i=1,\dots,s$. Hence
\[
d_i-d_{i-1}
=(a_i-a_{i+1})-(a_{i-1}-a_i)
=2a_i-a_{i+1}-a_{i-1}
=(2-b_i)a_i.
\]
Since $a_i>0$, the claim holds.
\qed

As a consequence of this lemma, we obtain the following.

\begin{corollary}\label{cor:LESR2}
   Assume that the LESR of $[b_1,\dots, b_s]$ is $(c_1,\dots,c_m)$. Then the LESR of $[b_1,\dots, b_s,2]$ is as follows:
\begin{itemize}
  \item If $b_s=2$, then $(c_1,\dots, c_{m-1}, c_m+1)$,
  \item if $b_s \geq 3 $, then $ (c_1,\dots, c_{m-1}, c_m,1)$.
\end{itemize}

\end{corollary}

\noindent
\proofname
If $b_s = 2$, then the last difference satisfies $d_s - d_{s-1} = 0$ by Lemma \ref{lem:LErule}. 
Hence the final block in the length encoding is extended, giving $(c_1,\dots,c_{m-1},c_m+1)$.
If instead $b_s \geq 3$, then $d_s - d_{s-1} \neq 0$, so a new block of length 
one appears at the end of the encoding  $ (c_1,\dots,c_{m-1},c_m,1)$.

\qed

\begin{lemma}\label{lemma:LESRsame}
  If the LESR of $[b_1,b_2, \dots, b_s]$ is $(c_1,\dots, c_m)$, then the LESR of $[b'_1, b_2, \dots, b_s]$ is also $(c_1,\dots, c_m)$.
\end{lemma}
\noindent
\proofname

Let $[b_2, \dots, b_s]=\dfrac{p}{q}$. Then we have
\[
[b_1,\dots, b_s ]=b_1-\dfrac{q}{p}=\dfrac{p \cdot b_1-q}{p}.
\]
Similarly, we have $ [b'_1, b_2, \dots, b_s]=\dfrac{p \cdot b'_1-q}{p}$. 
By Remark~3.2, the remainder sequences of these two continued fractions are
\[
b_1p-q,\ p,\ q,\ \dots,\ 1,\ 0
\quad\text{and}\quad
b'_1 p-q,\ p,\ q,\ \dots,\ 1,\ 0,
\]
respectively. Hence the indices of the special representations are $a_1=q,\ \dots,\ a_s=1$
in both cases. Therefore the corresponding difference sequences are identical, so the LESR are equal.

\qed

\subsection{Length encoding of $k$-Wahl chain}  
\begin{definition}\upshape \label{def:LEwahl}
  The \emph{length encoding} of a $k$-Wahl chain is defined as follows:
  \begin{itemize}
    \item[(i)] The empty sequence $\emptyset$ for $[k+2]$.
    \item[(ii)] Suppose that the sequence $(c_1,\dots,c_m)$ encodes a $k$-Wahl chain  $[a_1,\dots,a_s]$ with $a_1\neq 2$, then for any $c_{m+1} > 0$, $(c_1,\dots,c_{m+1})$ encodes the $k$-Wahl chain corresponding to 
 \[   [(2)^{c_{m+1}}, a_1, \dots , a_{s-1}, a_s+c_{m+1}], \]
      where $(2)^\ell$ denotes a string of $\ell$ consecutive $2$'s.
   \item[(iii)] Suppose that the sequence $(c_1,\dots,c_m)$ encodes a $k$-Wahl chain  $[a_1, \dots, a_s]$ with $a_s\neq 2$, then for any $c_{m+1} > 0$, $(c_1,\dots,c_{m+1})$ encodes the $k$-Wahl chain corresponding to
     $[a_1+c_{m+1}, a_2,\dots , a_{s}, (2)^{c_{m+1}}]$.
    \end{itemize}
  
\end{definition}

In other words, the length encoding of a $k$-Wahl chain records on which side successive 2’s are attached, starting from $[k+2]$. Moreover, the length encodings of $[b_1,\dots,b_s]$ and $[b_s,\dots,b_1]$ coincide, which does not affect the correspondence, since both represent the isomorphic cyclic quotient singularity.

\begin{example}\upshape
The length encoding of the $0$-Wahl chain $[4,2,2]$ is given by $(2)$.
The length encoding of $[2,4,2,3]$ is given by $(2,1)$, and that of $[3,2,4,2,4,2]$ is $(2,2,1)$.
\end{example}

\begin{theorem}
 Let $G=\tfrac{1}{r}(1,a)$ be a cyclic quotient singularity, where $r$ and $a$ are coprime.
 Suppose that the Hirzebruch–Jung continued fraction  
$r/a=[b_1,\dots,b_s]$ is a $k$–Wahl chain. 
Then the length encoding of $[b_1,\dots,b_s]$ agrees, up to reversal, 
with the length encoding of special representations {\rm (LESR)} of $G$. 
\end{theorem}

\noindent
\proofname


We proceed by induction on the number of steps in the inductive construction of a $k$-Wahl chain. 
Equivalently, we start from $[k+2]$ and at each step attach a $2$ either to the left or to the right. 
For the initial chain $[k+2]$, both the length encoding and the LESR are empty. 
At the first step, the chains $[k+3,2]$ and $[2,k+3]$ both have length encoding $(1)$, 
and a direct computation shows that their LESR is also $(1)$.

Assume that the claim holds for a $k$-Wahl chain $[b_1,\dots,b_s]$ whose length encoding and LESR are both $(c_1,\dots,c_m)$. 
By definition of a $k$-Wahl chain, one of $b_1$ and $b_s$ is equal to $2$.
By Definition~\ref{def:LEwahl}, the length encoding is unchanged under reversal, and the same is true for the LESR. Therefore, without loss of generality, we may assume that $b_s=2$. Then we have $b_1 \geq 3$.

\medskip
\noindent\textbf{Case 1:} We consider $[b_1+1,b_2,\dots,b_s,2]$.  
By Definition \ref{def:LEwahl}, its length encoding is $(c_1,\dots,c_{m-1},c_m+1)$.  
On the other hand, by Corollary \ref{cor:LESR2} and Lemma \ref{lemma:LESRsame}, its LESR is also $(c_1,\dots,c_{m-1},c_m+1)$.  

\medskip
\noindent\textbf{Case 2:} We consider $[2,b_1,b_2,\dots,b_s+1]$.  
By Definition \ref{def:LEwahl}, its length encoding is $(c_1,\dots,c_m,1)$.  
Considering instead the reversed chain $[b_s+1,b_{s-1},\dots,b_1,2]$, and applying Corollary \ref{cor:LESR2} with $b_1 \geq 3$, we obtain that the LESR is also $(c_1,\dots,c_m,1)$.  

\medskip
Thus in both cases the length encoding and the LESR coincide. 
This completes the induction step and proves the theorem.
\qed

Since the length encoding of a $k$–Wahl chain is invariant under reversal of the continued fraction, the two encodings in fact coincide.

\begin{example}\upshape
Let $G=\tfrac{1}{29}(1,12)$ be the cyclic group generated by $\tfrac{1}{29}(1,12)$. 
Since $\tfrac{29}{12} = [3,2,4,2]$, this is a $0$-Wahl chain, and its length encoding is $(2,1)$. 
On the other hand, by Example~\ref{ex29} the LESR of $G$ is also $(2,1)$. 
Thus they coincide.
\end{example}

\section{Zero continued fractions and P-resolutions}
A P-resolution of a cyclic quotient singularity $W=\compl{2}/G_{r,a}$ is a partial resolution $f: W' \to W$ such that $W'$ has only $T$-singularities and the canonical divisor $K_{W'}$ is $f$-ample. 
From the point of view of deformation theory, this is essential: by the theorem of Kollár and Shepherd-Barron \cite{KSB}, there is a bijection between the set of P-resolutions of $W$ and the irreducible components of the versal deformation space $\mathrm{Def}(W)$. 
Moreover, each component is dominated by the $\mathbb{Q}$-Gorenstein deformations of the corresponding P-resolution. 
Thus the existence of a P-resolution ensures the presence of a distinguished component of $\mathrm{Def}(W)$, and the combinatorics of zero continued fractions provides an explicit description of all such P-resolutions.

In this section, we study zero continued fractions for dual $k$-Wahl chains and relate
them to P-resolutions of cyclic quotient singularities.
By Theorem~\ref{thm:cf-P-resolution}, admissible zero continued fractions of the dual
Hirzebruch--Jung continued fraction are in one-to-one correspondence with
P-resolutions.
We then determine the minimal weight $\alpha([b_1,\dots,b_s])$ for dual $2$-, $1$-, and
$0$-Wahl chains, obtaining the values $0$, $1$, and $2$, respectively.
In particular, the case of dual $1$-Wahl chains leads to extremal P-resolutions.

\begin{definition}\upshape
A \emph{zero continued fraction} is a negative continued fraction $[a_1,\dots,a_s]$,
with integers $a_i \geq 1$, such that its value is equal to $0$.
\end{definition}

Every zero continued fraction can be obtained from the basic case $[1,1]$ by successive blow--ups.
The blow--up operation is defined as follows.  

For any consecutive pair $(u,v)$ in a continued fraction, replace it by
\[
[\ldots, u, v, \ldots]\;\longmapsto\;[\ldots, u{+}1,\, 1,\, v{+}1, \ldots].
\]

\medskip
If the chosen pair lies at the boundary, the same rule applies with the convention that
only the existing neighbor is increased.  
Thus
\[
[u, \ldots]\;\longmapsto\;[1,\, u{+}1, \ldots], 
\qquad
[\ldots, v]\;\longmapsto\;[\ldots, v{+}1,\, 1].
\]

\medskip
Starting from $[1,1]$ and iterating these blow--ups generates all zero continued fractions.  

\begin{example}\upshape
The following continued fractions are zero continued fractions:
\[
[1,1]\;\mapsto\;[1,2,1],\qquad
[1,2,1]\;\mapsto\;[1,3,1,2],\qquad
[1,2,1]\;\mapsto\;[1,2,2,1].
\]
  
\end{example}

It is a classical observation that zero continued fractions can be naturally interpreted in terms of triangulations of convex polygons. More precisely, consider a convex polygon with vertices $P_0,\dots, P_r$. By choosing a triangulation, each vertex $P_i$
 is contained in a certain number $v_i$
 of triangles. Then the sequence $[v_1,\dots,v_r]$ is a zero continued fraction, and, conversely, every zero continued fraction can be realized from some triangulation in this way. In Figure \ref{fig:my_label}, we have $v_0=3, v_1=1, v_2=2, v_3=2, v_4=1$. Therefore, $[1,2,2,1], [2,2,1,3],[2,1,3,1],[1,3,1,2], [3,1,2,2]$ are zero continued fractions.

\begin{figure}
  \centering
  \begin{tikzpicture}

  \coordinate(A) at ({2*cos(72)}, {2*sin(72)}) node at (A)[above]{$P_0$};
  \coordinate(B) at ({2*cos(144)}, {2*sin(144)}) node at (B)[above, left]{$P_4$};
  \coordinate(C) at ({2*cos(216)}, {2*sin(216)}) node at (C)[left]{$P_3$};
  \coordinate(D) at ({2*cos(288)}, {2*sin(288)}) node at (D)[below]{$P_2$};
  \coordinate(E) at ({2*cos(0)}, {2*sin(0)}) node at (E)[right]{$P_1$};

  \draw (A) -- (B) -- (C) -- (D)-- (E) -- cycle;

\draw (A)--(C);
\draw (A)--(D);

\end{tikzpicture}
  \caption{Zero continued fractions and polygon triangulations}
  \label{fig:my_label}
\end{figure}
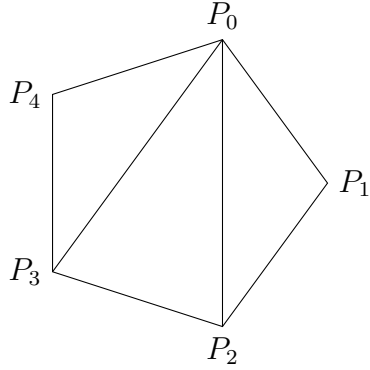

\begin{definition}\upshape \label{def:zcfweight}
Let $\frac{r}{a}=[b_1,\dots,b_s]$
be the Hirzebruch--Jung continued fraction, where $b_i \geq 2$ for all $i$. 
We say that $[b_1,\dots,b_s]$ admits a {zero continued fraction of weight $\lambda$} if there are indices $i_1 < i_2 < \cdots <i_u$ for some $u\geq1$ and integers $d_{i_j} \geq 1$ such that
\[
[\dots, b_{i_1}-d_{i_1}, \dots, b_{i_2}-d_{i_2}, \dots, b_{i_u}-d_{i_u}, \dots ] =0 ,
\]
and $\lambda+1 = \sum_{j=1}^{u}d_{i_j}$.
\end{definition}

\begin{definition}\upshape
Among all admissible zero continued fractions $[e_1,\dots,e_s]$ for $[b_1,\dots,b_s]$ 
and all possible choices of $(i_j,d_{i_j})$ as above, let 
\[
\alpha([b_1,\dots,b_s])
\]
denote the minimum of the resulting weights $\lambda$. 
We call $\alpha([b_1,\dots,b_s])$ the \emph{minimal weight of zero CF of $[b_1,\dots,b_s]$}.
When convenient, we also write $\alpha(r/a)$ for $[b_1,\dots,b_s]=r/a$.
\end{definition}

\begin{remark}\upshape
  Every Hirzebruch--Jung continued fraction $[b_1,\dots,b_s]$ with $s \geq 2$ has at least one zero continued fraction $[1,2,\dots,2,1]$. In particular, for an $A_n$ singularity, the admissible zero continued fraction is unique except when $n=3$. More precisely, if $n=3$, then there are exactly two such zero continued fractions, namely $[1,2,1]$ and $[2,1,2]$. 
\end{remark}

\begin{proposition}{\upshape \cite[Proposition 1.18]{U}}\label{prop:weight0condition}
 The Hirzebruch--Jung continued fraction $[b_1,\dots,b_s]$ has a zero continued fraction of weight $0$ if and only if its dual is a $T$-chain, i.e. the Hirzebruch-Jung continued fraction of $\frac{dm^2}{adm-1}$. In particular, a dual $2$-Wahl chain has a zero continued fraction of weight $0$.
\end{proposition}

\begin{proposition}{\upshape \cite[Proposition 1.19]{U}}
Let $0<q<m$ be coprime integers. If $\frac{m}{m-q} = [\,b_{1},\ldots,b_{s}\,]$ admits a
zero continued fraction of weight $1$, then the HJ continued fraction of $\frac{m}{q}$ is a
$T$-chain with $d=2$ and $n>1$, \emph{or}
\[
  \frac{m}{q} = [\,W_{0},\,c,\,W_{1}\,],
\]
where $c\ge 1$ and $W_i$ are $2$-Wahl chains (including one or both empty).
\end{proposition}

\begin{theorem}{\upshape \cite{Altmann, Christophersen}}\label{thm:cf-P-resolution}
Let $W=\compl{2}/G_{r,a}$ be a cyclic quotient singularity, and let $\frac{r}{r-a}=[b_1,\dots,b_s]$ be the dual Hirzebruch--Jung continued fraction.

Define
\[
K(W):=\left\{[k_1,\dots,k_s]=0 \;\middle|\; 1\le k_i\le b_i \text{ for all } i \right\}.
\]
Then there is a one-to-one correspondence between the set $K(W)$ and the set of P-resolutions of $W$. Equivalently, in the notation of Definition~\ref{def:zcfweight}, the admissible zero continued fractions
supported on $[b_1,\dots,b_s]$ are in one-to-one correspondence with the P-resolutions
of $W$.
\end{theorem}

\begin{theorem}\label{thm:k-wahlchain-zcf}
Let $[b_1,\dots,b_s]$ be a Hirzebruch--Jung continued fraction with $s\geq 2$. Then the following hold:
\begin{itemize}
  \item[{\rm (i)}] If it is a dual $1$-Wahl chain, then $\alpha([b_1,\dots,b_s])=1$.
  \item[{\rm(ii)}] If it is a dual $0$-Wahl chain, then $\alpha([b_1,\dots,b_s])=2$.
\end{itemize}
\end{theorem}
\noindent
\proofname

First, we prove (i). The dual $1$-Wahl chain $[2,2]$ admits a zero continued fraction $[1,1]$, whose weight is $1$. \\
We claim that if $[b_1,\dots,b_s]$ has a zero continued fraction $[v_1,\dots,v_s]$ which is induced by a triangulation of the convex polygon $P_0 P_1 \dots P_s$ with $v_0=1$, then $[2,b_1,\dots,b_{s-1},b_s+1]$ and $[b_1+1,b_2,\dots, b_s,2]$ have zero continued fractions $[2,v_1,\dots,v_{s-1},v_s+1]$ and $[v_1+1,v_2,\dots, v_s,2]$, respectively. Indeed, consider the polygon $P_0P_1\dots P_s$ with triangulation given by $(v_0,\dots,v_s)$. Adding a new vertex $P_{s+1}$ so that $P_0P_1\dots P_sP_{s+1}$ is convex, we obtain $v_{s+1}=1$, and both $v_0$ and $v_s$ increase by one compared to the original triangulation (see Figure \ref{fig:chainrule}). Hence the sequence $[v_0,v_1,\dots,v_{s-1},v_s+1] $, which equals $[2,v_1,\dots,v_{s-1},v_s+1]$, is a zero continued fraction. Thus $[2,b_1,\dots,b_{s-1},b_s+1]$ admits a zero continued fraction of weight $1$. Similarly, if we add the new vertex $P_{s+1}$ so that $P_0P_{s+1}P_1\dots P_s$ becomes convex, then $v_0$ and $v_1$ both increase by one, and we obtain $[v_1+1,v_2,\dots,v_s,v_0]$, that is $[v_1+1,v_2,\dots,v_s,2]$, as a zero continued fraction. By Proposition \ref{prop:weight0condition}, weight $0$ is impossible for a dual $1$-Wahl chain. 
Hence $\alpha=1$ for all dual $1$-Wahl chains.\\
Next we prove (ii). The dual $0$-Wahl chain $[3,2]$ admits a zero continued fraction $[1,1]$, and its
weight is $2$. Moreover, this zero continued fraction corresponds to a triangulation
with $v_0=1$.
By the same inductive argument as above, every dual $0$-Wahl chain admits a zero
continued fraction of weight $2$ arising from a triangulation with $v_0=1$.

It remains to show that this weight is minimal.
Let $[v_1,\dots,v_s]$ be any admissible zero continued fraction of $[b_1,\dots,b_s]$,
and write $d_i:=b_i-v_i$ for $1\le i\le s$.
Then, by Definition~4.4, we have
\[
\lambda+1=\sum_{i=1}^s d_i=\sum_{i=1}^s (b_i-v_i),
\]
Therefore, for fixed $[b_1,\dots,b_s]$, minimizing the weight $\lambda$ is equivalent
to maximizing $\sum_{i=1}^s v_i$.

On the other hand, if $[v_1,\dots,v_s]$ corresponds to a triangulation of the
$(s+1)$-gon $P_0P_1\cdots P_s$, then
\[
\sum_{i=0}^s v_i = 3(s-1).
\]
Since every triangulation satisfies $v_0\ge 1$, $\sum_{i=1}^s v_i$
is maximal if and only if $v_0=1$.
Now the weight $2$ zero continued fraction constructed above for a dual $0$-Wahl
chain is realized by a triangulation with $v_0=1$. Hence it maximizes
$\sum_{i=1}^s v_i$, and therefore its weight is minimal among all admissible zero
continued fractions. Therefore, we conclude that $\alpha([b_1,\dots,b_s])=2.$

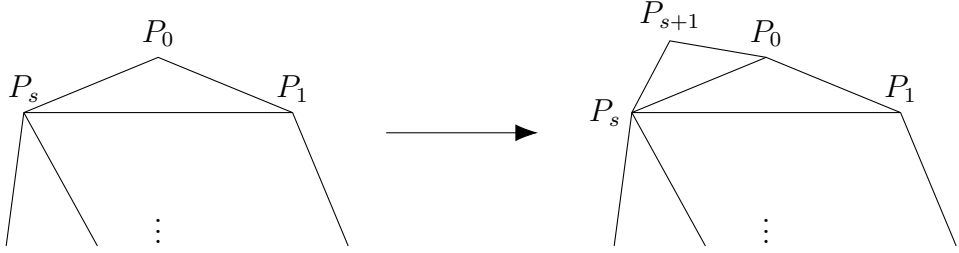
\begin{figure}
\centering

  \begin{tikzpicture}
\centering
  \coordinate(A) at ({2.5*cos(45)}, {2.5*sin(45)}) node at (A)[above]{$P_1$};
  \coordinate(B) at ({2.5*cos(90)}, {2.5*sin(90)}) node at (B)[above]{$P_0$};
  \coordinate(C) at ({2.5*cos(135)}, {2.5*sin(135)}) node at (C)[above]{$P_s$};
  \coordinate(D) at ({2*cos(180)}, {2.5*sin(180)});
  \coordinate(E) at ({2.5*cos(0)}, {2.5*sin(0)});
  \coordinate(F) at (-0.8,0);  
  \coordinate(O) at (0,0.3) node at (O){$\vdots$}; 
  \draw (A) -- (B) -- (C) -- (D) ;
  \draw (E) -- (A) ;
\draw (C) -- (F);
 
\draw (A)--(C);

\coordinate(1) at (3,1.5);
\coordinate(2) at (5,1.5);

\draw[-{Latex[length=3mm]}] (1) -- (2);

  \coordinate[xshift=8cm](I) at ({2.5*cos(45)}, {2.5*sin(45)}) node at (I)[above]{$P_1$};
  \coordinate[xshift=8cm](J) at ({2.5*cos(90)}, {2.5*sin(90)}) node at (J)[above]{$P_0$};
  \coordinate[xshift=8cm](K) at ({2.5*cos(135)}, {2.5*sin(135)}) node at (K)[left]{$P_s$};
  \coordinate[xshift=8cm](L) at ({2.0*cos(180)}, {2*sin(180)});
  \coordinate[xshift=8cm](M) at ({2.5*cos(0)}, {2.5*sin(0)});
  \coordinate[xshift=8cm](N) at (-0.8,0);  
  \coordinate[xshift=8cm](P) at (0,0.3) node at (P){$\vdots$}; 
 \coordinate[xshift=8cm](Q) at ({3*cos(115)}, {3*sin(115)}) node at (Q)[above]{$P_{s+1}$};
\draw (I) -- (J) -- (K) -- (L);
\draw (M) -- (I);
\draw (I) -- (K);
\draw (K) -- (N);

 \draw (K) -- (Q) -- (J);
\end{tikzpicture}
 
  \caption{Adding a new vertex}
  \label{fig:chainrule}
\end{figure}

\qed


We recall the notion of an extremal P-resolution. Extremal P-resolutions arise naturally in the birational geometry of degenerations of surfaces with Wahl singularities. In Urz\'ua's framework of \(W\)-surfaces, they appear as surface-theoretic models associated with semistable extremal neighborhoods and their flips. Thus, among all P-resolutions, extremal P-resolutions correspond to the case where the exceptional locus is as small as possible and all singularities on the partial resolution are Wahl singularities.

\begin{definition}\upshape
Let $f: W^+ \to W$ be a P-resolution of a cyclic quotient singularity
$(P \in W)$. We say that $f$ is an \emph{extremal P-resolution} if
\(f^{-1}(P)\) is a smooth rational curve \(\Gamma\), and \(W^+\) has only Wahl
singularities.
\end{definition}

\begin{remark}\label{rem:weight1-extremal}\upshape
Let $W=\compl{2}/G_{r,a}$ be a cyclic quotient singularity, and let $\frac{r}{r-a}=[b_1,\dots,b_s]$
be the dual Hirzebruch--Jung continued fraction.
By \cite[\S 4.1]{HTU}, the extremal P-resolutions of $W$ are in one-to-one
correspondence with pairs of indices $1\le i<j\le s$ such that
\[
[b_1,\dots,b_i-1,\dots,b_j-1,\dots,b_s]=0.
\]
In our terminology, this means precisely that the dual Hirzebruch--Jung continued
fraction admits an admissible zero continued fraction of weight $1$.
\end{remark}

Remark \ref{rem:dualkwahl}, Theorem \ref{thm:k-wahlchain-zcf} and Remark \ref{rem:weight1-extremal} lead to the following.

\begin{corollary}
If $r/a$ is a $1$-Wahl chain, then the cyclic quotient singularity $\compl{2}/G_{r,a}$ has an extremal P-resolution.
\end{corollary}

Yusuke Sato\\
\textsc{Department of Mathematics, Osaka Institute of Technology, 5-16-1 Ohmiya, Asahi-ku, Osaka, 535-8585, Japan}.\\
email: yusuke.sato@oit.ac.jp \\

\end{document}